\documentclass[11pt]{amsart}

\usepackage[a4paper,margin=3cm]{geometry}
\usepackage{amsmath,amssymb,mathtools}
\usepackage{enumitem}
\usepackage[hidelinks]{hyperref}

\numberwithin{equation}{section}

\newtheorem{theorem}{Theorem}[section]
\newtheorem{proposition}[theorem]{Proposition}
\newtheorem{lemma}[theorem]{Lemma}

\newcommand{\Tr}{\operatorname{Tr}}
\newcommand{\diag}{\operatorname{diag}}

\newcommand{\norm}[1]{\left\lVert #1\right\rVert}

\newcommand{\C}{\mathbb{C}}

\title[A negative exponent range for a complementary McCarthy inequality]
{A negative exponent range for Audenaert's complementary McCarthy trace inequality}

\author{Xing Li}
\address{Center for Control Theory and Guidance Technology, Harbin Institute of
Technology, Harbin 150001, China}
\email{pipipi1919810@gmail.com}

\author{Bin Zhou$^{\ast}$} 
\address{Center for Control Theory and Guidance Technology, Harbin Institute of
Technology, Harbin 150001, China}
\email{binzhou@gmail.com; binzhou@hit.edu.cn}

\thanks{This paper was supported in part by the National Natural Science
Foundation of China for Distinguished Young Scholars under Grant 62125303; by
the National Natural Science Foundation of China -- ``Qisun Ye'' Science
Foundation under Grant U2441243; by the National Natural Science Foundation of
China under Grants 62573165 and 62521005; by the Natural Science Foundation of
Heilongjiang Province under Grant FG2025F001; and by the Fundamental and
Interdisciplinary Disciplines Breakthrough Plan of the Ministry of Education of
China JYB2025XDXM206.}
\thanks{$\ast$ Corresponding author}


\begin{document}

\begin{abstract}
Audenaert introduced a class of complementary McCarthy type trace inequalities in his work on completely monotone functions and Bernstein functions; the same problem was later included in the problem list of Audenaert and Kittaneh.  The known results cover the corresponding directions for \(q\le -2\), \(0<q\le 1\), \(1\le q\le 2\), and \(2\le q\le 3\), while the negative range \(-2<q<0\) was left open at the time.  We prove the negative exponent inequality, in fact for every \(q<0\).  After the inversion \(X=A^{-1}\), \(Y=B^{-1}\), the problem reduces to a trace inequality for the parallel sum \(X:Y\).  The proof uses the Kubo--Ando mean chain, an Ando--Hiai type log-majorization for matrix geometric means, and a finite-dimensional Schatten H\"older inequality, including the quasi-norm range.  For positive exponents, we give a finite-dimensional proof for every real \(q>3\) based on a Hessian estimate for \(\Tr|Y|^3\); together with Audenaert's result for \(2\le q\le3\), this covers all \(q\ge2\).  Finally, we determine the equality case on the negative side: equality holds if and only if \(A=B\).
\end{abstract}

\maketitle

\section{Introduction}

Let \(A,B\) be positive matrices.  McCarthy's trace inequality says that
\begin{equation}\label{eq:mccarthy}
    \Tr(A+B)^q\ge \Tr A^q+\Tr B^q,\qquad q\ge 1,
\end{equation}
while the reverse inequality holds for \(0<q\le 1\); see McCarthy's work on Schatten classes \cite{McCarthy1967}.  In his study of trace inequalities for completely monotone functions and Bernstein functions, Audenaert obtained complementary estimates involving the mixed term
\[
    \Tr(A^{1/2}BA^{1/2})^{q/2}.
\]
For \(q\ge2\), these estimates strengthen McCarthy's inequality; for negative exponents, they give estimates for inverse powers.

The complementary McCarthy type inequalities in question can be written as follows.  For \(q<0\), and for positive definite \(A,B\), the expected direction is
\begin{equation}\label{eq:direct-q}
    \Tr(A+B)^q
    \le
    \Tr A^q+\Tr B^q
    +(2^q-2)\Tr(A^{1/2}BA^{1/2})^{q/2}.
\end{equation}
For positive exponents, the corresponding inequality in the opposite direction is
\begin{equation}\label{eq:reverse-q}
    \Tr(A+B)^q
    \ge
    \Tr A^q+\Tr B^q
    +(2^q-2)\Tr(A^{1/2}BA^{1/2})^{q/2}.
\end{equation}
Audenaert's Corollary 3 in \cite{Audenaert2012Trace} establishes the upper-bound direction in \eqref{eq:direct-q} for \(q\le -2\) and \(1\le q\le 2\), and the lower-bound direction in \eqref{eq:reverse-q} for \(0<q\le 1\) and \(2\le q\le 3\).  The problem list of Audenaert and Kittaneh \cite[Section~3, Conjecture~1, p.~19]{AudenaertKittaneh2017Problems} then singled out the range \(-2<q<0\) and the positive range \(q\ge 3\) in its original conjecture.  Audenaert later noted that the duality argument in \cite{Audenaert2006NormCompression} for \(q>3\) was incorrect \cite{Audenaert2012Trace}. We prove this range in Section \ref{sec:positive} by a finite-dimensional Hessian argument.

Our main result is the following.

\begin{theorem}\label{thm:main}
Let \(A,B\) be positive matrices.
\begin{enumerate}[label=\textup{(\roman*)}]
    \item If \(A,B\) are positive definite, then \eqref{eq:direct-q} holds for every \(q<0\).
    \item If \(A,B\) are positive semidefinite, then \eqref{eq:reverse-q} holds for every \(q\ge 2\).
\end{enumerate}

\end{theorem}

Part \textup{(i)} is proved in Section \ref{sec:negative}. Part \textup{(ii)} follows from Audenaert's result for \(2\le q\le3\) and the argument for \(q>3\) in Section \ref{sec:positive}.

For the negative exponent part, put \(q=-r\) and make the inversion
\[
    X=A^{-1},\qquad Y=B^{-1}.
\]
Then
\[
    (A+B)^{-1}=X:Y=(X^{-1}+Y^{-1})^{-1},
\]
where \(X:Y\) denotes the parallel sum.  The main estimate is a trace inequality for \(X:Y\).  We first prove a one-parameter asymmetric form in which the mixed term
\[
    \norm{X^{1/2}Y^{1/2}}_r^r
\]
is replaced by
\[
    \norm{X^{1-t}Y^t}_r^r,\qquad 0<t<1.
\]
The symmetric value \(t=1/2\) gives the sharp estimate needed for \eqref{eq:direct-q}.

Throughout, \(A,B\) and \(X,Y\) appearing in the same inequality belong to the same \(M_n(\mathbb C)\), unless otherwise stated. The notation \(Z>0\) and \(Z\ge0\) denotes positive definiteness and positive semidefiniteness, respectively. The trace is the usual non-normalized trace.  For \(p>0\) and a matrix \(Z\), we write
\[
    \norm{Z}_p=(\Tr |Z|^p)^{1/p}.
\]
When \(0<p<1\), this is the Schatten \(p\)-quasi-norm.

\section{Preliminaries}\label{sec:prelim}

For \(X,Y>0\), the parallel sum is
\[
    X:Y=(X^{-1}+Y^{-1})^{-1}.
\]
We use the Kubo--Ando harmonic and geometric means
\[
    X!Y=2(X:Y),\qquad
    X\#Y=X^{1/2}(X^{-1/2}YX^{-1/2})^{1/2}X^{1/2}.
\]
More generally, for \(0<t<1\), set
\[
    X!_tY=((1-t)X^{-1}+tY^{-1})^{-1},
    \qquad
    X\#_tY=X^{1/2}(X^{-1/2}YX^{-1/2})^tX^{1/2}.
\]
These means satisfy the standard weighted mean chain
\begin{equation}\label{eq:weighted-mean-chain}
    X!_tY\le X\#_tY\le (1-t)X+tY.
\end{equation}
At \(t=1/2\) this reduces to
\begin{equation}\label{eq:mean-chain}
    X!Y\le X\#Y\le \frac{X+Y}{2}.
\end{equation}
These are the harmonic, geometric, and arithmetic Kubo--Ando mean inequalities.  See Kubo--Ando \cite{KuboAndo1980} and Bhatia \cite{Bhatia2007}.

For a matrix \(Z\), let \(s(Z)=(s_1(Z),\ldots,s_n(Z))\) be its singular values in decreasing order.  For a positive matrix \(W\), let \(\lambda(W)\) be its eigenvalue vector in decreasing order.  For nonnegative vectors $x=(x_j)_{j=1}^n$ and $y=(y_j)_{j=1}^n$
with entries arranged in decreasing order, we write $x\prec_{\log}y$ if
$\prod_{j=1}^k x_j\le\prod_{j=1}^k y_j$ for $1\le k<n$ and
$\prod_{j=1}^n x_j=\prod_{j=1}^n y_j$.
We write $x\prec_w y$ if $\sum_{j=1}^k x_j\le
\sum_{j=1}^k y_j$ for all $1\le k\le n$. We also use
$x\prec_{w\log}y$ for the analogous product inequalities for every
$1\le k\le n$ without requiring equality at $k=n$.

\begin{lemma}\label{lem:ando-hiai}
Let \(X,Y>0\) and \(0<t<1\).  Then
\begin{equation}\label{eq:AH-weighted}
    \lambda(X\#_tY)\prec_{\log}s(X^{1-t}Y^t).
\end{equation}
Consequently, for every \(r>0\),
\begin{equation}\label{eq:geomean-trace-bound-weighted}
    \Tr(X\#_tY)^r\le \norm{X^{1-t}Y^t}_r^r.
\end{equation}
In particular, when \(t=1/2\),
\begin{equation}\label{eq:geomean-trace-bound}
    \Tr(X\#Y)^r\le \norm{X^{1/2}Y^{1/2}}_r^r.
\end{equation}
\end{lemma}

\begin{proof}
Ando and Hiai \cite[Corollary~2.3, p.~120]{AndoHiai1994}, with
$H=\log X$, $K=\log Y$, $\alpha=t$ and $p=q=1$, give, for $0<t<1$,
the following consequence of their two log-majorization inequalities
(see also \cite[Corollary~5.7, p.~40]{SeoWadaYamazaki2026}):
\[
\lambda(X\#_tY)\prec_{\log}
\lambda\bigl(\exp((1-t)\log X+t\log Y)\bigr)\prec_{\log}
\lambda\bigl(X^{(1-t)/2}Y^tX^{(1-t)/2}\bigr).
\]
To justify the remaining passage, put $C=X^{1-t}$ and $D=Y^t$.
The positive matrix $C^{1/2}DC^{1/2}$ and the product $CD$
have the same eigenvalues. Weyl's multiplicative majorization
of eigenvalue moduli by singular values (e.g.~\cite[Theorem~II.3.6, pp.~42--43]{Bhatia1997}) therefore yields
\[
 \lambda\bigl(C^{1/2}DC^{1/2}\bigr)
 = |\lambda(CD)|\prec_{\log}s(CD).
\]
Combining these relations proves \eqref{eq:AH-weighted}.

If \(x\prec_{\log}y\) and the entries are non-negative and arranged in decreasing order, then for every \(r>0\), the vector \((x_1^r,\ldots,x_n^r)\) is weakly majorized by \((y_1^r,\ldots,y_n^r)\).  Hence
\[
    \sum_j \lambda_j(X\#_tY)^r
    \le
    \sum_j s_j(X^{1-t}Y^t)^r,
\]
which is \eqref{eq:geomean-trace-bound-weighted}.  Taking \(t=1/2\) gives \eqref{eq:geomean-trace-bound}.
\end{proof}

\begin{lemma}\label{lem:trace-monotone}
If \(0\le C\le D\), then for every \(r>0\),
\[
    \Tr C^r\le \Tr D^r.
\]
Moreover, if equality holds, then \(C=D\).
\end{lemma}

\begin{proof}
By the min-max principle, the eigenvalues in decreasing order satisfy
\[
    \lambda_j(C)\le \lambda_j(D),\qquad j=1,\ldots,n.
\]
Since \(x\mapsto x^r\) is increasing on \([0,\infty)\) for \(r>0\), taking powers and using \(\Tr C^r=\sum_j\lambda_j(C)^r\) gives \(\Tr C^r\le \Tr D^r\).

If \(\Tr C^r=\Tr D^r\), then every scalar inequality
\[
    \lambda_j(C)^r\le \lambda_j(D)^r
\]
must be an equality.  Thus \(\Tr C=\Tr D\).  Since \(D-C\ge 0\), we get
\[
    \Tr(D-C)=0,
\]
and hence \(D-C=0\).
\end{proof}

\begin{lemma}\label{lem:holder}
Let \(p,q,r>0\) and \(1/r=1/p+1/q\).  For all compatible matrices \(U,V\),
\[
    \norm{UV}_r\le \norm{U}_p\norm{V}_q.
\]
In particular, for \(X,Y>0\), \(r>0\), and \(0<t<1\),
\begin{equation}\label{eq:holder-weighted-special}
    \norm{X^{1-t}Y^t}_r^r
    \le
    (\Tr X^r)^{1-t}(\Tr Y^r)^t.
\end{equation}
At \(t=1/2\),
\begin{equation}\label{eq:holder-special}
    \norm{X^{1/2}Y^{1/2}}_r^r
    \le
    (\Tr X^r)^{1/2}(\Tr Y^r)^{1/2}
    \le
    \frac{\Tr X^r+\Tr Y^r}{2}.
\end{equation}
\end{lemma}

\begin{proof}
This is the finite-dimensional Schatten--H\"older inequality, including the quasi-norm range; see \cite[Theorem~6.20, pp.~246--247]{HiaiPetz2014}.

We first prove the general form.  In finite dimensions, Horn's singular value product inequality gives the weak log-majorization
\[
    s(UV)\prec_{w\log} s(U)s(V),
\]
where singular value vectors are padded with zeros to a common length and the product on the right is componentwise; see Bhatia \cite[Theorem~IV.2.5, p.~94]{Bhatia1997}.  Raising entries to the power \(r\), summing, and using the summation property of weak majorization gives
\[
    \norm{UV}_r^r
    =\sum_j s_j(UV)^r
    \le
    \sum_j s_j(U)^r s_j(V)^r.
\]
Since \(1/r=1/p+1/q\), we have \(p/r>1\), \(q/r>1\), and
\[
    \frac{1}{p/r}+\frac{1}{q/r}=1.
\]
The scalar H\"older inequality yields
\[
    \sum_j s_j(U)^r s_j(V)^r
    \le
    \left(\sum_j s_j(U)^p\right)^{r/p}
    \left(\sum_j s_j(V)^q\right)^{r/q}
    =
    \norm{U}_p^r\norm{V}_q^r.
\]
Taking \(r\)-th roots proves \(\norm{UV}_r\le \norm{U}_p\norm{V}_q\) for all the stated positive exponents.

Now take \(U=X^{1-t}\), \(V=Y^t\), \(p=r/(1-t)\), and \(q=r/t\).  Then
\[
    \norm{X^{1-t}Y^t}_r
    \le
    \norm{X^{1-t}}_{r/(1-t)}
    \norm{Y^t}_{r/t}.
\]
Since \(X,Y\) are positive,
\[
    \norm{X^{1-t}}_{r/(1-t)}=(\Tr X^r)^{(1-t)/r},
    \qquad
    \norm{Y^t}_{r/t}=(\Tr Y^r)^{t/r}.
\]
Raising to the power \(r\) gives \eqref{eq:holder-weighted-special}.  The first inequality in \eqref{eq:holder-special} is the case \(t=1/2\), and the second one is the scalar arithmetic-geometric mean inequality.
\end{proof}

\section{A trace inequality for parallel sums}\label{sec:parallel}

The negative exponent inequality will follow from a trace estimate for the parallel sum.  We first prove a one-parameter asymmetric form.  For \(0<t<1\), set
\[
    m_t=\max\{t,1-t\}.
\]

\begin{theorem}\label{thm:parallel-asym}
Let \(X,Y>0\), \(r>0\), and \(0<t<1\).  Then
\begin{equation}\label{eq:parallel-asym}
    \Tr(X:Y)^r+
    \bigl(m_t^{-1}-m_t^r\bigr)\norm{X^{1-t}Y^t}_r^r
    \le
    \Tr X^r+\Tr Y^r .
\end{equation}
\end{theorem}

\begin{proof}
Put
\[
    H=X:Y,\qquad
    G_t=\norm{X^{1-t}Y^t}_r^r,\qquad
    S=\Tr X^r+\Tr Y^r.
\]
Since
\[
    (1-t)X^{-1}+tY^{-1}
    \le
    m_t(X^{-1}+Y^{-1}),
\]
inverting gives
\[
    X:Y\le m_t(X!_tY).
\]
By \eqref{eq:weighted-mean-chain},
\[
    X:Y\le m_t(X\#_tY).
\]
Using Lemmas \ref{lem:trace-monotone} and \ref{lem:ando-hiai}, we get
\begin{equation}\label{eq:asym-H-estimate}
    \Tr H^r
    \le
    m_t^r\Tr(X\#_tY)^r
    \le
    m_t^r G_t.
\end{equation}
On the other hand, Lemma \ref{lem:holder}, in the form \eqref{eq:holder-weighted-special}, gives
\begin{equation}\label{eq:asym-holder}
    G_t
    =
    \norm{X^{1-t}Y^t}_r^r
    \le
    (\Tr X^r)^{1-t}(\Tr Y^r)^t.
\end{equation}
The scalar weighted arithmetic-geometric mean inequality yields
\[
    (\Tr X^r)^{1-t}(\Tr Y^r)^t
    \le
    (1-t)\Tr X^r+t\Tr Y^r
    \le
    m_tS.
\]
Hence
\begin{equation}\label{eq:asym-G-estimate}
    m_t^{-1}G_t\le S.
\end{equation}
Combining \eqref{eq:asym-H-estimate} and \eqref{eq:asym-G-estimate}, we obtain
\[
    \Tr H^r+
    \bigl(m_t^{-1}-m_t^r\bigr)G_t
    \le
    m_t^rG_t+
    \bigl(m_t^{-1}-m_t^r\bigr)G_t
    =
    m_t^{-1}G_t
    \le S.
\]
This is \eqref{eq:parallel-asym}.
\end{proof}

\begin{theorem}\label{thm:parallel}
Let \(X,Y>0\) and \(r>0\).  Then
\begin{equation}\label{eq:parallel-theorem}
    \Tr(X:Y)^r+
    \left(2-2^{-r}\right)\norm{X^{1/2}Y^{1/2}}_r^r
    \le
    \Tr X^r+\Tr Y^r.
\end{equation}
The coefficient \(2-2^{-r}\) is optimal, in the sense that it cannot be increased uniformly over all positive definite \(X,Y\).
\end{theorem}

\begin{proof}
Take \(t=1/2\) in Theorem \ref{thm:parallel-asym}.  Then \(m_t=1/2\) and
\[
    m_t^{-1}-m_t^r=2-2^{-r},
\]
which gives \eqref{eq:parallel-theorem}.

If \(X=Y\), then \(X:Y=X/2\) and
\[
    \norm{X^{1/2}Y^{1/2}}_r^r=\Tr X^r.
\]
Thus \eqref{eq:parallel-theorem} is an equality at \(X=Y\).  Increasing the coefficient would therefore make the inequality fail at this point.
\end{proof}

\section{The negative exponent range}\label{sec:negative}

We now derive the negative exponent part of Theorem \ref{thm:main}.

\begin{theorem}\label{thm:negative}
Let \(A,B>0\) and \(q<0\).  Then
\[
    \Tr(A+B)^q
    \le
    \Tr A^q+\Tr B^q+
    (2^q-2)\Tr(A^{1/2}BA^{1/2})^{q/2}.
\]
\end{theorem}

\begin{proof}
Write \(q=-r\), where \(r>0\), and set
\[
    X=A^{-1},\qquad Y=B^{-1}.
\]
Then
\begin{equation}\label{eq:AB-parallel}
    (A+B)^{-1}=(X^{-1}+Y^{-1})^{-1}=X:Y.
\end{equation}
Moreover,
\[
    (A^{1/2}BA^{1/2})^{-1}
    =
    A^{-1/2}B^{-1}A^{-1/2}
    =
    X^{1/2}YX^{1/2}.
\]
The positive matrices \(X^{1/2}YX^{1/2}\) and \(Y^{1/2}XY^{1/2}\) have the same nonzero eigenvalues, and
\[
    Y^{1/2}XY^{1/2}
    =
    (X^{1/2}Y^{1/2})^*(X^{1/2}Y^{1/2}).
\]
Therefore
\begin{equation}\label{eq:mixed-transform}
    \Tr(A^{1/2}BA^{1/2})^{-r/2}
    =
    \norm{X^{1/2}Y^{1/2}}_r^r.
\end{equation}
Applying Theorem \ref{thm:parallel} to \(X,Y\), and then using \eqref{eq:AB-parallel} and \eqref{eq:mixed-transform}, gives
\[
    \Tr(A+B)^{-r}
    +
    \left(2-2^{-r}\right)
    \Tr(A^{1/2}BA^{1/2})^{-r/2}
    \le
    \Tr A^{-r}+\Tr B^{-r}.
\]
Since \(q=-r\), this is the desired inequality.
\end{proof}

\section{The positive exponent range}\label{sec:positive}

The range $2\le q\le3$ follows from Audenaert
\cite[Corollary~3]{Audenaert2012Trace}. We treat every real $q>3$ by a
finite-dimensional Hessian argument.

For a twice continuously differentiable real function $f$, put
$F_f(X)=\Tr f(X)$ for Hermitian $X$. The second directional derivative is
written $D^2F_f(X)[H,H]$.

\begin{lemma}\label{lem:positive-cubic}
Let $Y=Y^*$ and let $H=\diag(h_1,\dots,h_n)$ be real diagonal. Then
\begin{equation}\label{eq:positive-cubic}
 D^2\Tr|Y|^3[H,H]\ge 6\sum_{i=1}^n h_i^2|Y_{ii}|.
\end{equation}
\end{lemma}
\begin{proof}
Write $Y=P-Q$, where $P=Y_+$ and $Q=Y_-$, so $P,Q\ge0$ and $PQ=0$.
Let $(a_\alpha,u_\alpha)$ and $(b_\beta,v_\beta)$ denote the positive
spectral pairs of $P$ and $Q$. For a polynomial $f$, differentiating $\Tr f(Y+tH)$ twice and using
cyclicity of the trace gives, in an eigenbasis of $Y$,
\[
 D^2\Tr f(Y)[H,H]=\sum_{j,k}f'^{[1]}(\lambda_j,\lambda_k)|H_{jk}|^2,
\]
where $f'^{[1]}(x,y)=(f'(x)-f'(y))/(x-y)$ for $x\ne y$ and
$f'^{[1]}(x,x)=f''(x)$. For a $C^2$ function on a compact spectral
interval, approximate $f''$ uniformly by polynomials and integrate twice,
fixing the value and first derivative at one point. The functions and
first two derivatives converge uniformly. The divided differences converge
uniformly as well, since
$f'^{[1]}(x,y)=\int_0^1f''((1-t)y+tx)\,dt$.
The polynomial derivative formulas therefore converge uniformly on a
neighborhood of $Y$, proving the formula also for $f\in C^2$.
Applying it to $f(x)=|x|^3$ gives
\begin{equation}\label{eq:positive-cubic-identity}
 \frac16D^2\Tr|Y|^3[H,H]
 =\Tr((P+Q)H^2)
 -2\sum_{\alpha,\beta}\frac{a_\alpha b_\beta}{a_\alpha+b_\beta}
  |\langle u_\alpha,Hv_\beta\rangle|^2.
\end{equation}
For $f'(x)=3x|x|$, this divided difference equals
$3(|x|+|y|)$ for $xy\ge0$, and for $x=a>0$, $y=-b<0$ equals
$3(a+b)-6ab/(a+b)$. This proves \eqref{eq:positive-cubic-identity}.
The divided-difference convention includes repeated and zero eigenvalues.

Set $p_i=P_{ii}$ and $r_i=Q_{ii}$. Split $H=H_1+H_2$ into real diagonal
matrices by
\[
 (H_1)_{ii}=\frac{r_i}{p_i+r_i}h_i,\qquad
 (H_2)_{ii}=\frac{p_i}{p_i+r_i}h_i
\]
when $p_i+r_i>0$; if $p_i+r_i=0$, assign that diagonal entry to either
summand. The weighted Cauchy--Schwarz inequality
\[
 \frac{ab}{a+b}|x+y|^2\le a|x|^2+b|y|^2
\]
gives, on writing $\Pi_+$ and $\Pi_-$ for the support projections of $P$
and $Q$,
\begin{align*}
 \sum_{\alpha,\beta}a_\alpha
 |\langle u_\alpha,H_1v_\beta\rangle|^2
 &=\sum_\alpha a_\alpha\|\Pi_-H_1u_\alpha\|^2
 \le\Tr(PH_1^2),\\
 \sum_{\alpha,\beta}b_\beta
 |\langle u_\alpha,H_2v_\beta\rangle|^2
 &=\sum_\beta b_\beta\|\Pi_+H_2v_\beta\|^2
 \le\Tr(QH_2^2).
\end{align*}
Consequently,
\begin{align*}
 &\sum_{\alpha,\beta}\frac{a_\alpha b_\beta}{a_\alpha+b_\beta}
  |\langle u_\alpha,Hv_\beta\rangle|^2\\
 &\quad\le\Tr(PH_1^2)+\Tr(QH_2^2)
 =\sum_i h_i^2\frac{p_i r_i}{p_i+r_i}.
\end{align*}
Hence the left side of \eqref{eq:positive-cubic-identity} is at least
$\sum_i h_i^2(p_i^2+r_i^2)/(p_i+r_i)$, which is at least
$\sum_i h_i^2|p_i-r_i|=\sum_i h_i^2|Y_{ii}|$.
If $p_i+r_i=0$, positivity implies that the corresponding rows and columns of $P$ and $Q$ vanish; all fractions at that index are interpreted as zero.
\end{proof}

\begin{lemma}\label{lem:positive-hessian}
Let $X>0$ and $H=\diag(h_1,\dots,h_n)$ be real diagonal. If
$f\in C^2([0,\infty))$ and $f''$ is convex, then
\begin{equation}\label{eq:positive-hessian}
 D^2\Tr f(X)[H,H]\ge\sum_i h_i^2 f''(X_{ii}).
\end{equation}
\end{lemma}
\begin{proof}
For $s\ge0$ let $\phi_s(x)=(x-s)_+^3/6$. Since
$12\phi_s(x)=(x-s)^3+|x-s|^3$, Lemma~\ref{lem:positive-cubic}
and the polynomial identity
\[
 D^2\Tr(X-sI)^3[H,H]=6\Tr((X-sI)H^2)
 =6\sum_i h_i^2(X_{ii}-s)
\]
yield
\begin{equation}\label{eq:positive-hinge}
 D^2\Tr\phi_s(X)[H,H]\ge\sum_i h_i^2(X_{ii}-s)_+.
\end{equation}
On a compact interval containing the spectrum of $X$ and all $X_{ii}$,
interpolate $f''$ by convex piecewise-linear functions $g_m$ with mesh
tending to zero. Each $g_m$ is an affine function plus a finite
nonnegative linear combination of $(x-s)_+$. Integrating twice gives a
cubic polynomial plus the corresponding nonnegative linear combination
of the functions $\phi_s$. The estimate is an identity for cubic
polynomials and follows from \eqref{eq:positive-hinge} for every hinge.
Since $g_m\to f''$ uniformly, the twice-integrated functions converge
to $f$ in $C^2$ (with their value and first derivative fixed at one
point). The spectral divided-difference formula above shows continuity
of the trace Hessian under $C^2$ convergence. Passing to the limit
proves the assertion.
\end{proof}

\begin{theorem}\label{thm:positive-block}
Let $A\in M_m(\mathbb C)$ and $B\in M_n(\mathbb C)$ be positive semidefinite, let $C\in\mathbb C^{m\times n}$, and assume
$M=\begin{pmatrix}A&C\\ C^*&B\end{pmatrix}\ge0$.
For every real $q>3$,
\begin{equation}\label{eq:positive-block}
 \Tr M^q\ge\Tr A^q+\Tr B^q+(2^q-2)\Tr|C|^q,
 \qquad |C|=(C^*C)^{1/2}.
\end{equation}
\end{theorem}
\begin{proof}
More generally, let $D,H$ be Hermitian with $D\pm H\ge0$.
Diagonalize $H=\diag(h_i)$ and put $d_i=D_{ii}\ge|h_i|$.
For $f(x)=x^q$, the second derivative
$f''(x)=q(q-1)x^{q-2}$ is convex on $[0,\infty)$ when $q>3$.
Apply Lemma~\ref{lem:positive-hessian} to $D+tH+\varepsilon I$
for $-1\le t\le1$. These matrices are positive definite because
$D+tH=\frac{1+t}{2}(D+H)+\frac{1-t}{2}(D-H)\ge0$.
For $G_\varepsilon(t)=\Tr(D+tH+\varepsilon I)^q$, use
\[
 G_\varepsilon(1)+G_\varepsilon(-1)-2G_\varepsilon(0)
 =\int_{-1}^1(1-|t|)G_\varepsilon''(t)\,dt.
\]
Lemma~\ref{lem:positive-hessian} gives
\[
 G_\varepsilon''(t)\ge
 q(q-1)\sum_i h_i^2(d_i+th_i+\varepsilon)^{q-2}.
\]
The right-hand side is the sum of the second derivatives of
\((d_i+th_i+\varepsilon)^q\). Integrating against \(1-|t|\) gives
the sum of their centered second differences as a lower bound.
Since \(d_i\ge|h_i|\), continuity of the powers on \([0,\infty)\)
allows us to let $\varepsilon\downarrow0$, giving
\begin{align}\label{eq:positive-center}
 &\Tr(D+H)^q+\Tr(D-H)^q-2\Tr D^q\notag\\
 &\quad\ge\sum_i\big[(d_i+h_i)^q+(d_i-h_i)^q-2d_i^q\big].
\end{align}
For fixed $h$, put $\psi_h(d)=(d+h)^q+(d-h)^q-2d^q$.
For $d\ge|h|$,
\[
 \psi_h'(d)=q\big[(d+h)^{q-1}+(d-h)^{q-1}-2d^{q-1}\big]\ge0
\]
by convexity of $x^{q-1}$. Thus $\psi_h(d)\ge\psi_h(|h|)=(2^q-2)|h|^q$.
Consequently the left side of \eqref{eq:positive-center} is at least
$(2^q-2)\Tr|H|^q$.

Take $D=\diag(A,B)$ and
$H=\begin{pmatrix}0&C\\C^*&0\end{pmatrix}$.
The matrices $D+H=M$ and $D-H$ are unitarily similar, and
$\Tr|H|^q=2\Tr|C|^q$. The result follows after division by two.
\end{proof}

\begin{proposition}\label{prop:positive}
For all positive semidefinite matrices $A,B$ of the same order and all
real $q\ge2$, inequality \eqref{eq:reverse-q} holds.
\end{proposition}
\begin{proof}
The range $2\le q\le3$ is due to Audenaert
\cite[Corollary~3]{Audenaert2012Trace}. For $q>3$, apply
Theorem~\ref{thm:positive-block} with $C=A^{1/2}B^{1/2}$.
Put $V=\binom{A^{1/2}}{B^{1/2}}$. Then the block matrix is
$VV^*\ge0$, whereas $V^*V=A+B$, so their nonzero eigenvalues agree.
Moreover, $C^*C=B^{1/2}AB^{1/2}$ and $CC^*=A^{1/2}BA^{1/2}$ have
the same nonzero eigenvalues. Hence
$\Tr|C|^q=\Tr(A^{1/2}BA^{1/2})^{q/2}$, as required.
\end{proof}

\section{Equality and sharpness}\label{sec:sharpness}

We first determine the equality case in the negative exponent range.

\begin{proposition}\label{prop:equality-negative}
Let \(A,B>0\) and \(q<0\).  Then
\[
    \Tr(A+B)^q
    =
    \Tr A^q+\Tr B^q+
    (2^q-2)\Tr(A^{1/2}BA^{1/2})^{q/2}
\]
if and only if \(A=B\).
\end{proposition}

\begin{proof}
If \(A=B\), direct substitution gives equality.

Conversely, put \(q=-r\), \(X=A^{-1}\), and \(Y=B^{-1}\).  If equality holds in the displayed formula, then equality must hold in the estimates used to prove Theorem \ref{thm:parallel-asym} at \(t=1/2\).  More explicitly, with
\[
    H=X:Y,\qquad
    G=\norm{X^{1/2}Y^{1/2}}_r^r,\qquad
    S=\Tr X^r+\Tr Y^r,
\]
the proof used
\[
    \Tr H^r\le 2^{-r}\Tr(X\#Y)^r\le 2^{-r}G,
    \qquad
    2G\le S.
\]
Equality in the final result forces equality throughout these two chains.  In particular, from
\[
    H=X:Y\le \frac12(X\#Y)
\]
and Lemma \ref{lem:trace-monotone}, we get
\[
    X:Y=\frac12(X\#Y).
\]
Let \(Z=X^{-1/2}YX^{-1/2}\).  By the Kubo--Ando representations, the last equality is equivalent to
\[
    I:Z=\frac12 Z^{1/2}.
\]
On the spectrum of \(Z\), this says
\[
    \frac{z}{1+z}=\frac12 z^{1/2},\qquad z>0.
\]
Equivalently, \((z^{1/2}-1)^2=0\).  Thus \(Z=I\), so \(X=Y\), and hence \(A=B\).
\end{proof}

The coefficient \(2^q-2\) is forced by the symmetric normalization.  Indeed, if \(A=B>0\), then
\[
    \Tr(A+B)^q=2^q\Tr A^q,
    \qquad
    \Tr(A^{1/2}BA^{1/2})^{q/2}=\Tr A^q.
\]
Therefore
\[
    \Tr A^q+\Tr B^q+
    (2^q-2)\Tr(A^{1/2}BA^{1/2})^{q/2}
    =
    2^q\Tr A^q.
\]
Thus both directions in Theorem \ref{thm:main} are equalities when \(A=B\).  On the negative side, since \(2^q-2<0\), the coefficient cannot be decreased.  On the positive side, the coefficient cannot be increased.

For positive numbers \(a,b\), let
\[
    \Phi_q(a,b)
    =
    (a+b)^q-a^q-b^q-(2^q-2)(ab)^{q/2}.
\]
By homogeneity, one may take \(a=e^t\), \(b=e^{-t}\).  Then
\begin{equation}\label{eq:scalar-normalization}
    \Phi_q(e^t,e^{-t})
    =
    2^q(\cosh^q t-1)-2(\cosh(qt)-1).
\end{equation}
Near \(t=0\),
\begin{equation}\label{eq:scalar-expansion}
    \Phi_q(e^t,e^{-t})
    =
    q(2^{q-1}-q)t^2+O(t^4).
\end{equation}
This local expansion is consistent with the signs in Theorem \ref{thm:main}: the quadratic coefficient is negative for \(q<0\) and non-negative for \(q\ge 2\), with equality at \(q=2\).

\section{Conclusion}

Theorem \ref{thm:negative} and Proposition \ref{prop:positive} together prove Theorem \ref{thm:main}. The negative exponent range \(q<0\), in particular the interval \(-2<q<0\) in the Audenaert--Kittaneh problem list, follows from the symmetric parallel sum inequality \eqref{eq:parallel-theorem}. Equality in this range holds exactly when \(A=B\). For \(q>3\), the positive exponent inequality follows from the cubic Hessian estimate and the block-matrix argument in Section \ref{sec:positive}.

After the submission of this paper, it came to our attention that Zhang \cite[Theorem~1.2]{Zhang2026II} also established the complementary McCarthy trace inequality for positive definite matrices in both the positive and negative exponent ranges considered here. Zhang's Remark~1.3 cites our result for all negative exponents.

\section*{Declarations}
\begin{enumerate}
\item \textbf{Ethical Declarations and Approval}: The authors declare that they have no competing interests relevant to the content of this article.
\item \textbf{Competing Interests}: The authors declare that they have no conflicts of interest.
\item \textbf{Authors' Contributions}: Both authors contributed equally to the paper. The corresponding author, Bin Zhou, performed the final editing, which both authors approved.
\item \textbf{Availability of Data and Materials}: Data sharing is not applicable to this article, as no datasets or other materials were generated or analyzed during the current study.
\end{enumerate}
\bibliographystyle{plain}
\bibliography{references}

@article{AndoHiai1994,
  author = {Ando, T. and Hiai, F.},
  title = {Log majorization and complementary {Golden--Thompson} type inequalities},
  journal = {Linear Algebra and its Applications},
  volume = {197/198},
  year = {1994},
  pages = {113--131}
}

@article{Audenaert2006NormCompression,
  author = {Audenaert, K. M. R.},
  title = {A norm compression inequality for block partitioned positive semidefinite matrices},
  journal = {Linear Algebra and its Applications},
  volume = {413},
  number = {1},
  year = {2006},
  pages = {155--176}
}

@article{Audenaert2012Trace,
  author = {Audenaert, K. M. R.},
  title = {Trace inequalities for completely monotone functions and {Bernstein} functions},
  journal = {Linear Algebra and its Applications},
  volume = {437},
  number = {2},
  year = {2012},
  pages = {601--611}
}

@article{AudenaertKittaneh2017Problems,
  author = {Audenaert, Koenraad M. R. and Kittaneh, Fuad},
  title = {Problems and conjectures in matrix and operator inequalities},
  journal = {Banach Center Publications},
  volume = {112},
  year = {2017},
  pages = {15--31},
  doi = {10.4064/bc112-0-2}
}

@book{Bhatia1997,
  author = {Bhatia, R.},
  title = {Matrix Analysis},
  series = {Graduate Texts in Mathematics},
  volume = {169},
  publisher = {Springer},
  address = {New York},
  year = {1997}
}

@book{Bhatia2007,
  author = {Bhatia, R.},
  title = {Positive Definite Matrices},
  series = {Princeton Series in Applied Mathematics},
  publisher = {Princeton University Press},
  address = {Princeton},
  year = {2007}
}

@article{KuboAndo1980,
  author = {Kubo, F. and Ando, T.},
  title = {Means of positive linear operators},
  journal = {Mathematische Annalen},
  volume = {246},
  year = {1980},
  pages = {205--224}
}

@article{McCarthy1967,
  author = {McCarthy, C. A.},
  title = {$C_p$},
  journal = {Israel Journal of Mathematics},
  volume = {5},
  year = {1967},
  pages = {249--271},
  doi = {10.1007/BF02771613}
}

@book{HiaiPetz2014,
  author    = {Hiai, Fumio and Petz, D{\'e}nes},
  title     = {Introduction to Matrix Analysis and Applications},
  series    = {Universitext},
  publisher = {Springer},
  address   = {Cham},
  year      = {2014},
  doi       = {10.1007/978-3-319-04150-6}
}

@article{SeoWadaYamazaki2026,
  author  = {Seo, Yuki and Wada, Shuhei and Yamazaki, Takeaki},
  title   = {On the {Ando--Hiai} property for spectral geometric means},
  journal = {Linear Algebra and its Applications},
  volume  = {746},
  pages   = {18--47},
  year    = {2026},
  doi     = {10.1016/j.laa.2026.05.013}
}

@misc{Zhang2026II,
  author = {Zhang, Teng},
  title = {{Clarkson--McCarthy} type inequalities, part {II}: refinements, unitary orbits, unitarily invariant norms and {$\ell_s$--$\ell_s$} estimates},
  year = {2026},
  howpublished = {arXiv:2611.00459v1},
  note = {Preprint, submitted 30 October 2026},
  doi = {10.48550/arXiv.2611.00459}
}

\end{document}